\documentclass[preprint,11pt]{article}

\usepackage{amssymb,amsfonts,amsmath,amsthm,amscd,dsfont,mathrsfs}
\usepackage{graphicx,float,psfrag,epsfig}
\usepackage{wrapfig,hyperref}
\usepackage{relsize}

\usepackage{halffull}

\DeclareMathAlphabet{\mathpzc}{OT1}{pzc}{m}{it}

\newtheorem{propo}{Proposition}[section]
\newtheorem{theorem}[propo]{Theorem}
\newtheorem{lemma}[propo]{Lemma}
\newtheorem{corollary}[propo]{Corollary}
\newtheorem{conjecture}[propo]{Conjecture}
\newtheorem{fact}[propo]{Fact}
\newtheorem{definition}[propo]{Definition}
\newtheorem{remark}[propo]{Remark}

\def\codim{\operatorname{codim}}
\def\dim{\operatorname{dim}}

\def\vol{\operatorname{vol}}

\def\mean{\operatorname{\mathbb{E}}}
\def\span{\operatorname{span}}
\def\dist{\operatorname{dist}}
\def\ker{\operatorname{Ker}}
\def\normal{{\sf N}}
\def\I{{\sf I}}
\def\reals{\mathbb{R}}
\def\ve{\epsilon}
\def\vr{\operatorname{vr}}
\def\sign{\operatorname{sign}}
\def\F{\mathcal{F}}
\def\H{\mathcal{H}}

\def\r{{\rm rad}}

\def\size#1{|#1|}

\def\P{\mathcal{P}}
\def\E{\mathbb{E}}
\long\def\commented#1{{}}
\def\note#1{\textbf{[[#1]]}}

\date{\today}

\title{The minimax risk of truncated series estimators for\\
symmetric convex polytopes}
\author{Adel Javanmard\thanks{Department of Electrical Engineering, Stanford University, CA. The work was partially done during an internship at Microsoft Research.} \and 
Li Zhang\thanks{Microsoft Research Silicon Valley, Mountain View, CA}}

\begin{document}

\maketitle

\begin{abstract}
We study the optimality of the minimax risk of truncated series
estimators for symmetric convex polytopes. We show that the optimal
truncated series estimator is within $O(\log m)$ factor of the optimal
if the polytope is defined by $m$ hyperplanes.  This represents the
first such bounds towards general convex bodies.  In proving our
result, we first define a geometric quantity, called the
\emph{approximation radius}, for lower bounding the minimax risk.  We
then derive our bounds by establishing a connection between the
approximation radius and the Kolmogorov width, the quantity that
provides upper bounds for the truncated series estimator.  Besides,
our proof contains several ingredients which might be of independent
interest: 1. The notion of approximation radius depends on the volume
of the body. It is an intuitive notion and is flexible to yield strong minimax
lower bounds; 2. The connection between the approximation radius and
the Kolmogorov width is a consequence of a novel duality relationship
on the Kolmogorov width, developed by utilizing some deep results from
convex geometry~\cite{bt-87,st-89,g-95}.
\end{abstract}

\section{Introduction}\label{sec:intro}

In this paper, we study the minimax risk of estimators for symmetric
convex polytopes. We show that for a symmetric convex polytope defined
by $m$ hyperplanes, the truncated series estimator, a special type of
linear estimator, is within $O(\log m)$ factor of the optimal.

In non-parametric statistics, the minimax risk of an estimator
measures the worst case expected loss 
of the estimator for input
coming from some subset $X\subseteq\reals^n$ (see Section~\ref{sec:minimax risk} for a formal definition). Tremendous work has been
done on understanding the optimal minimax risk for
various families of $X$. But it is usually very difficult to design
the optimal estimator.  The truncated series estimator is a family of
linear estimator that simply projects an observation to a properly
chosen subspace.  Despite its simplicity, the truncated series
estimator is surprisingly powerful and is shown to be nearly optimal for
wide families of convex bodies. \cite{pinsker-80} shows that such
estimator is nearly optimal for ellipsoids.  In~\cite{dlm-90}, it is
shown that it is nearly optimal for the wider family of orthosymmetric
and quadratically convex objects, including $\ell_p$ balls for $p\geq
2$. 

In this paper, we show that the power of truncated series estimator
extends to the rich class of symmetric polytopes. Specifically, we
show that for a symmetric convex polytope defined by $m$ hyperplanes,
the truncated series estimator is within $O(\log m)$ factor of the
optimal. Previously, such results have only been obtained for
particular family of convex polytopes, such those corresponding to the
Lipschitz condition~\cite{nemirovski-98} or satisfying certain isometric conditions~\cite{rwy-11}.
As a motivating example, we discuss one application of our result in
estimating values of a Lipschitz function. 

\smallskip

\noindent{\bf Example.} One important estimation problem in the literature is the
estimation of functions satisfying certain continuity or Lipschitz
conditions from noisy measurements. Consider a univariate Lipschitz function $f:[0,1]\to
\reals$. Suppose that $x_i = f(t_i)$ for $i=1, \ldots, n$, and we have measurements $y_i$ according to the model $y_i=x_i+w_i$ for some
gaussian noise $w_i$.  Then Lipschitz condition, with constant $L$,
translates to the linear constraints:
\begin{equation}\label{eq:lip}
\mbox{$|x_{i+1} - x_i|\leq L\, |t_{i+1} - t_i|$, for $i=1, \ldots, n-1$.}
\end{equation}
 
Now, we are interested in estimating $x_i$ from $y_i$. A key
observation is that the vector $x = (x_1,\cdots,x_n)$ falls in the set
$X$, where
\begin{eqnarray}
X = \{x : |x_{i+1} - x_i| \le L\, |t_{i+1} - t_i|,\text{ for } 1\le i\le n-1\}.
\end{eqnarray}
Note that $X$ is a symmetric convex polytope.
\smallskip\\

When the sampling is uniform, i.e. $t_i = (i-1)/(n-1)$, then $X$ has a more special form of $X=\{x: |x_{i+1}-x_i| \leq L/(n-1)\}$.  In this case, previous work~\cite{nemirovski-98,tsybakov-09} has shown that the best
truncated series estimator is nearly optimal.  
As a consequence of our work, the truncated series estimator is nearly optimal (within $O(\log n)$ factor) for estimating Lipschitz function at arbitrary sample set $\{t_1,\ldots,t_n\}$. 

At the high level, the proof of our results follows a very simple
strategy. We choose a family of ``obstruction objects'' for which we
can obtain lower bounds of the minimax risk. Then we show a
``duality'' result that if $X$ does not have a good truncated series
estimator, then it will have to contain a ``large'' obstruction, and
therefore no estimator can do well on $X$.  Of course, the difficulty
is in choosing the obstruction so that we can prove the corresponding
duality result.  Some natural obstructions include 
hyper-rectangles and Euclidean balls, for which we know very tight
minimax lower bound. But they turn out to be too restrictive to allow
a strong enough duality result.  To overcome this difficulty, we
consider a broader family consisting of objects which contain a ``non-negligible'' fraction of a ``large'' 
Euclidean ball; whence we are able to establish a desired duality
relationship. 

More specifically, we first define a geometric measure for any set,
called \emph{approximation radius}, and then develop a lower bound
technique which bounds the minimax risk of any body by its
approximation radius. Intuitively, the approximation radius of an
object $X$ is the maximum radius of a ball with ``non-negligible''
volume fraction inside $X$.  By refining the technique
in~\cite{yb-99}, we can show that the minimax risk of $X$ is
asymptotically as large as that of the ball with $X$'s approximation
radius (see Theorem~\ref{thm:lowerbound}).  On the other hand, the
minimax risk of truncated series estimator is determined by the
Kolmogorov width of the object. Our bound is then derived by
establishing a connection between the Kolmogorov widths and the
approximation radius of the symmetric convex bodies (see
Theorem~\ref{thm:approxradius}).  For the connection, we first derive
a duality relationship between the Kolmogorov widths of $X$ and its
polar dual $X^\circ$ (see Theorem~\ref{thm:dual}), by utilizing some
results from convex geometry started in~\cite{bt-87}.
The Kolmogorov width of $X^\circ$ is then shown, by probabilistic
arguments, to be intimately related to the approximation radius of
$X$.

\subsection{Related work}

There is a vast body of work on the minimax estimators and it is
beyond the scope of this paper to survey all of them. We refer
to~\cite{nemirovski-98,tsybakov-09,johnstone-11} for comprehensive
survey and will describe some work most relevant to this paper.
Since we focus on the mean squared error (MSE), all the subsequent discussion is in the context of MSE.

The minimax bounds have been developed for various families of convex
bodies through intensive research in the past decades.
Asymptotically tight bounds have been proposed for convex
bodies that correspond to various continuity or energy conditions; the classes of H\"{o}lder balls, Sobolev balls, and Besov
balls. We refer to Chapter~2.8 in~\cite{tsybakov-09} for a
comprehensive recount of the references.  Despite these remarkable results, it
is still largely unknown how to compute the minimax risk for an arbitrary
convex body.  Some previous work does attempt to deal with less specific objects (see~\cite{rwy-11} and the references therein), but all the optimality results are under (fairly strong) isometric assumption about the objects. 

On the other hand, the truncated series estimator has a nice geometric
interpretation and is related to the classical Kolmogorov width of the
underlying space.  In addition to its simplicity, \cite{dlm-90}~shows that it
is asymptotically optimal for the classes of orthosymmetric and
quadratically convex objects. This includes the class of diagonally
stretched $\ell_p$ balls for $p\geq 2$. Present paper shows that the power of truncated series
estimators also extend to the family of symmetric convex polytopes, as long as
the polytope is defined by $n^{O(1)}$ hyperplanes. 

To achieve our result, we develop a lower bound technique based on a
geometric quantity which we dub approximation radius. Using
Fano's inequality and the refinement developed in~\cite{yb-99,rwy-11},
we show that the minimax risk of a convex body is lower bounded by
that of the ball with radius equal to the approximation radius of that body. Compared to the
existing lower bound techniques, such as the Bernstein bound and the
bound followed from considering the worst (typically discrete) distributions
(see~\cite{nemirovski-98,tsybakov-09} and~\cite{dj-94,djmm-11}), the
approximation radius relies on a volume estimation and is both
convenient to operate and flexible to provide strong lower bounds.

One center piece in this paper is the connection established between the
approximation radius and the Kolmogorov width. Towards this step, we use 
some results developed in Banach space geometry which was
initialized in~\cite{bt-87} for investigating the invertibility of
matrices with large ``robust'' rank and subsequently developed
by~\cite{st-89,g-95}.  In particular, we show a duality relationship
between the Kolmogorov widths of a convex body and its polar dual body.
Our result has a similar flavor to the classical duality
in~\cite{m-90} but is tighter when the dimension gap is small.

\section{Preliminaries}

\subsection{Notations and definitions}

For a vector $x = (x_1,\cdots,x_n)$ and a real number $p \geq 1$,
denote by $\|x\|_p$ the $\ell_p$-norm of $x$, and $\|x\|_{\infty} =
\max_i |x_i|$.  When $p$ is absent, it means $\ell_2$ norm. Let
$B^n_p(x,r)$ denote the $n$ dimensional $\ell_p$ ball with radius $r$ and center
$x$. Whenever the center is at the origin, it is denoted by
$B^n_p(r)$. Also, we drop the superscript $n$, whenever the dimension
is clear from the context, and suppress the argument $r$ for $r=1$.

A set $X \subset \reals^n$ is called \emph{centrally symmetric} (or
simply symmetric) if for any $x \in X$, we have $-x\in X$.  For a set
$K$, the ($\ell_2$) radius of $X$ is defined as in the
following.
\begin{eqnarray*}
\begin{split}
\r(X) &= \max_{x\in K} \|x\|.
\end{split}
\end{eqnarray*}

For $p > 0$, and $n \ge 1$, the family $\F^{m,n}_p$ is defined as
\begin{eqnarray}\label{def:F}
\F^{m,n}_p = \{X:\, X = \{x:\; |Ax|_p\leq 1\},\mbox{for $A\in \reals^{m \times n}$} \}
\end{eqnarray}

In particular, when $p=\infty$, $\F^{m,n}_\infty$ consists of
symmetric convex polytopes defined by $m$ hyperplanes. Throughout we consider bounded convex bodies. Our results
easily extend to unbounded convex bodies, but the presentation would
be cumbersome by including separate case analysis which does not add
any new insight.

\subsection{Minimax risk}
\label{sec:minimax risk}

Suppose we are given measurements of an unknown $n$-dimensional vector $x$, according to the model
\begin{eqnarray}\label{eqn:model}
y= x+ w,
\end{eqnarray}
where $w \in \reals^n$ follows the normal distribution, $w \sim \normal(0,\sigma^2 \I)$, and $x$ lies in $X$, a compact convex set in $\reals^n$. The goal of the minimax estimation problem is to estimate vector $x$, with small error loss, and to evaluate the estimator under the minimax principle. 

For any estimator $M:R^n\to R^n$, the maximum mean squared error of $M$ on $(X,\sigma)$ is defined as
\begin{eqnarray*}
R(M, X, \sigma) =  \max_{x\in X} \mean \|x-M(y)\|^2 \,,
\end{eqnarray*}
and the minimax risk of $X$ is
\begin{eqnarray*}
R(X,\sigma) = \min_{M} R(M,X,\sigma)\,.
\end{eqnarray*}

Estimators generally can be  nonlinear function. We denote by $R_L(X,\sigma)$ the minimax risk when $M$ is linear. An alternative to the linear and nonlinear estimators is the truncated series estimator~\cite{dlm-90}. Truncated series estimator is obtained using projections $M(y) = Py$, with $P^2= P$. Throughout this paper, projection always mean orthogonal projection. The minimax risk for truncated estimators is defined as 
\[R_T(X,\sigma) = \min_{P} \max_{x\in X} \|x - Py\|^2,\]
where the minimum is taken over all the linear projections. \commented{For the case that $X$ is a hypercube with side lengths $\tau_i$, $1\le i \le n$, it is straightforward to see that the optimal truncated series estimator becomes a hard thresholding estimator. In fact, in this case the estimation $M(y)$ is obtained by considering $y_i$ as the estimate of $x_i$ \emph{only} in coordinates at which $\tau_i > \sigma$ and considering $0$ as the estimate of $x_i$ in the other coordinates~\cite{dlm-90}. }
Since truncated series estimators are linear, we clearly have
\begin{eqnarray*}
R(X,\sigma) \le R_L(X,\sigma) \le R_T(X,\sigma).
\end{eqnarray*}

It turns out that the minimax risk for truncated series estimators is completely  characterized by the Kolmogorov $k$-width $d_k$ of $X$, defined as~\cite{p-84}
\begin{eqnarray*}
d_k(X) = \min_{P_k} \max_{x\in X} \|x - P_k x\|,
\end{eqnarray*}
where the minimum is taken over all $k$-dimensional projections. Then, we have

\begin{equation}\label{eqn:upperbound}
R_T(X,\sigma)= \min_{k} \, d_k(X)^2 + k \sigma^2\,.
\end{equation}

For the mean squared error considered in this paper, there is a more direct equivalent definition of the Kolmogorov $k$-width under $\ell_2$ metric.
\begin{eqnarray*}
d_k(X) = \min_{P\in\P_k} \r(P(X))\,,
\end{eqnarray*}
where $\P_k$ denotes all the $k$-codimensional (or $n-k$ dimensional) projections, and $\r(K)$
denotes the $\ell_2$ radius of $K$, defined as $\max_{x\in K} \|x\|_2$. Furthermore,
\begin{eqnarray}\label{eqn:monotone_dk}
\r(X)= d_0(X)\geq d_1(X)\geq\ldots \geq d_n(X)=0\,.
\end{eqnarray}

\subsection{Approximation radius}

We define the notion of approximation radius, a geometric measure of any convex body, which as we shall show, provides a lower bound for the minimax risk of the body.

We use $\vol(X)$ to denote the volume of $X$ and $\H^k_n$ to denote
all the $k$ dimensional subspaces in $\reals^n$.  Assume $X\subseteq \reals^n$ is a convex
body that contains the origin. For any $r>0$, the volume ratio
$\vr(X,r)$ of $X$ is defined as 
\[\vr(X,r) = \left(\frac{\vol(X\cap B^n_2(r))}{\vol(B^n_2(r))}\right)^{1/n}\,,\]

and the $k$-volume ratio $\vr_k(X,r)$ of $X$ is defined as the maximum volume
ratio over all the $k$ dimensional central cut of $X$, i.e.
\[\vr_k(X,r) = \max_{H\in\H^k_n} \vr(X\cap H, r)\,.\]

Clearly, $0\leq \vr(X,r)\leq 1$. Further,
\begin{fact}\label{fact:volumeratio}
If $X$ is convex and contains the origin, then $\vr(X,r)$, and hence
$\vr_k(X,r)$ for any $k$, is non-increasing in $r$.
\end{fact}
\begin{proof}
It suffices to show for any $c>1$, $\vr(X,c\cdot r)\leq \vr(X,r)$. 
\[X \cap B^n_2(c\cdot r) = c (\frac{1}{c}X \cap B^n_2(r))\subseteq c(X\cap B^n_2(r))\,,\]
where $\frac{1}{c}X\subseteq X$ follows from the assumption that $X$ is convex and
contains the origin. Therefore $\vol(X\cap B^n_2(c\cdot r)) \leq c^n
\vol(X\cap B^n_2(r))$. The claim follows immediately from the
definition of volume ratio and the identity $\vol(B^n_2(c\cdot r)) = c^n
\vol(B^n_2(r))$.
\end{proof}

Central to lower bounding the minimax risk is the notion of
approximation radius.
\begin{definition}
For $0\leq c\leq 1$, and integer $1\le k \le n$, the $(c,k)$-\emph{approximation radius} of $X$, denoted by $z_{c,k}(X)$, is defined as the
maximum $r$ such that $\vr_k(X,r)\geq c$, i.e.
\begin{equation}\label{eq:zk_def}
z_{c,k}(X) = \sup \{r:\; \vr_k(X,r)\geq c \}\,.
\end{equation}
\end{definition}

Note that if $X$ contains the origin in its interior, then
$z_{c,k}(X)$ is always defined for $0\leq c\leq 1$.


\subsection{Polar dual of convex bodies}

The connection between the Kolmogorov width and the approximation radius is established via the polar dual of the body. We state some basic facts about the polar dual body which we will need later. 

\begin{definition}
For any $K\subset R^n$, denote by $K^\circ$ the (polar) dual set of
$K$,
\begin{eqnarray*}
K^\circ=\{ y\,|\, x\cdot y\leq 1\quad\mbox{for all $x\in K$}\}\,.
\end{eqnarray*}
If $K$ lies on a lower dimensional subspace, $K^\circ$ is understood as the dual set on the lowest dimensional subspace that contains $K$.
\end{definition}

\begin{fact}\label{fact:dual0}
If $X=\{x:\;|Ax|_\infty\leq 1\}$, then $X^\circ = A^T B_1^m$. 
\end{fact}

\begin{fact}\label{fact:dual}
Let $H$ be a subspace of $R^n$. Denote by $P_H$ the 
projection on $H$. Then $P_H(K^\circ) = (H\cap K)^\circ$.
\end{fact}
\begin{proof}
We include a proof of this fact for the sake of completeness.
We prove the different but equivalent identity $P_H(K)^\circ = H\cap
K^\circ$.  Let $m=\dim(H)$ and $H=(h_1, \ldots,
h_m)$ be an orthonormal basis of $H$. With a slight abuse of notation, we denote by $H$ the matrix in $\reals^{n\times m}$ that has $h_i$ as columns. Then $P_H(x)=HH^Tx$. Observe that for any $x\in R^n, y\in R^m$, $(HH^Tx) \cdot (Hy) = x^THH^THy = x^THy = x\cdot (Hy)$. Hence 
\begin{eqnarray*}
Hy\in P_H(K)^\circ & \Leftrightarrow &\forall x\in K, (HH^Tx) \cdot (Hy)\leq 1\\
&\Leftrightarrow&\forall x\in K, x\cdot Hy\leq 1\\
&\Leftrightarrow& Hy \in H\cap K^\circ\,.
\end{eqnarray*}
\end{proof}
%

\section{Main results}

In this paper, we are interested in the minimax risk of the truncated
series estimator for symmetric convex bodies. Define $\beta(X) =
\max_{\sigma>0} R_T(X,\sigma)/R(X,\sigma)$, and $\beta^{m,n}_p =
\max_{X\in \mathcal{F}^{m,n}_p} \beta(X)$.  Our main result is

\begin{theorem}\label{thm:linear}
If $n=\Omega(\log m)$, then $\beta^{m,n}_\infty \leq c\cdot\log m$,
where $c <  2\cdot 10^8$. Furthermore, $\beta^{m,n}_\infty=\Omega(\sqrt{\log
m/\log\log m})$.
\end{theorem}

The lower bound follows immediately from previous works. As shown
in~\cite{dj-94} (Theorem 3), for the unit $\ell_1$ ball $X = B^n_1$,
$R_T(X,1/\sqrt{n})=\Omega(1)$ but $R(X,1/\sqrt{n})=O(\sqrt{\log
n/n})$. Since $B^n_1\in F^{m,n}_\infty$ where $m=2^n$, we have
$\beta^{m,n}_\infty=\Omega(\sqrt{\log m/\log\log m})$ for
$n=\Omega(\log m)$.  In this paper, our main result is to provide a
nearly matching upper bound of $O(\log m)$. The upper bound is the
consequence of the following theorems: Theorem~\ref{thm:lowerbound}
lower bounds the minimax risk by the approximation radius;
Theorems~\ref{thm:dual}, \ref{thm:approxradius} together establish a lower bound
on the approximation radius by the Kolmogorov width, which in turn upper bounds the minimax risk of the truncated series estimator. We assign concrete
values to constants whenever possible. They are purely for
presentation clarity and by no means represent the best possible
constants.

\begin{theorem}\label{thm:lowerbound}
There exists a universal constant $C = 2.46\cdot 10^{-4}$ such
that for any $0<c_*\leq 1$,
\begin{equation}\label{eq:lowerbound}
R(X,\sigma)\ge C c_*^2\max_k \min\,\{z_{c_*,k}(X)^2, k \sigma^2\}\,.
\end{equation}
\end{theorem}

\begin{theorem}\label{thm:dual}
For any convex centrally symmetric $X\subset R^n$ and any $0\leq k\leq n$ and $0<\epsilon<1$,
\begin{eqnarray}
d_{k}(X) d_{n-(1-\epsilon)k} (X^\circ) \leq c_1 \sqrt{\frac{k}{\epsilon}}\,,
\end{eqnarray}
where $c_1 = 2/(\sqrt{2}-1)\leq 5$.
\end{theorem}

\begin{theorem}\label{thm:approxradius}
Let $X\in\F^{m,n}_\infty$. For any $0 < c_* \leq 0.2$ and $0<k\leq n$,
\begin{eqnarray}
z_{c_*,k}(X) \geq c_2 \sqrt{\frac{k}{\ln m}} \cdot \frac{1}{d_{n-k}(X^\circ)}\,,
\end{eqnarray}
where 
\begin{eqnarray}\label{eq:c2}
c_2 = 0.4 \sqrt{\ln(1/(2c_*))}\,.
\end{eqnarray}
\end{theorem}

The paper is mainly devoted to proving
Theorems~\ref{thm:lowerbound}, \ref{thm:dual}, and \ref{thm:approxradius}, which together imply
Theorem~\ref{thm:linear}. We discuss some consequences of our
results as well as some open questions at the end. 

\section{Lower bounding the minimax risk}\label{sec:lower}

In this section we prove Theorem~\ref{thm:lowerbound}. Our starting
point is from the obvious lower bound for the Euclidean ball
$B^n_2(r)$. It is well known that $R(B^n_2(r),\sigma)=\Omega(\min(r^2,
n\sigma^2))$.  We shall show that this is also true for any subset
contained in $B^n_2(r)$ with ``non-negligible'' (a fraction of
$\Omega(c^n)$ for some constant $c>0$) volume.

The proof is based on the information-theoretic bound established
in~\cite{yb-99}. In this technique, the minimax risk is lower bounded
by restricting to a maximal finite set of points $\{x_1,\cdots,x_r\}$
in $X$, separated from each other by at least an amount $\ve$ in the
loss metric. Indeed, $\ve$ is the maximum separation distance such
that the hypothesis $\{x_1,\cdots,x_r\}$ are almost
indistinguishable. The Fano inequality is then used to relate this
indistinguishability to K-L divergence.

We proceed by defining an $\ve$-net and a $\delta$-packing in a set
$S$.
\begin{definition}
A set $N_{\epsilon} \subseteq S$ is said to be an
\emph{$\epsilon$-net} for $S$ if for any $x \in S$, there exists a
$x_0 \in N_{\epsilon}$, such that $\|x- x_0\| \le \epsilon$. In
addition, a finite set $M_{\delta} \subseteq S$ is said to be an
$\delta$-packing in $S$, if for any $x,x' \in M_{\delta}$, $x \neq
x'$, we have $\|x-x'\| > \delta$.
\end{definition}

\begin{propo}\label{pro:Fano_mod}
For any set $X$, let $N_\epsilon(X)$ be any $\epsilon$-net for $X$ and
$M_\delta(X)$ be a $\delta$-packing in $X$. Then,
\begin{eqnarray}\label{eqn:Fano_mod}
R(X,\sigma) \geq \left(\frac{\delta}{2}\right)^2 \left(1-\frac{\log
\size{N_\epsilon(X)} + \frac{\epsilon^2}{2\sigma^2} + 1}{\log
\size{M_\delta(X)}}\right).
\end{eqnarray}
\end{propo} 

Proposition~\ref{pro:Fano_mod} is a direct application of the bound
proved in~\cite{yb-99} (Theorem~1). For the reader's convenience, we
give the details of its derivation in Appendix~\ref{app:Fano_mod}.

Note that the strongest lower bound in Eq.~\eqref{eqn:Fano_mod} is achieved per the smallest $\ve$-net and the largest $\delta$-packing of $X$. In the following, we will develop an upper bound on the size of the smallest $\ve$-net for $X$ and a lower bound for the size of its largest $\delta$-packing.

\begin{lemma}\label{lem:net}
For any $X\subseteq \reals^n$, $r \ge \r(X)$ and $\epsilon \leq r$, there exists an $\epsilon$-net for $X$, with size at most $(3r / \epsilon)^n$. 
\end{lemma}
The proof of Lemma~\ref{lem:net} is deferred to Appendix~\ref{app:net}.

\begin{lemma}\label{lem:pack}
For any $\delta>0$, there exists a $\delta$-packing $M_\delta(X)$ with size at least 
$\frac{\vol(X)}{\vol(B_2(\delta))}$. 
\end{lemma}
We refer to Appendix~\ref{app:pack} for the proof of Lemma~\ref{lem:pack}.

We are now in position to prove Theorem~\ref{thm:lowerbound}.
\begin{proof}(Theorem~\ref{thm:lowerbound})
For any $k$ and $c_*$ consider the $k$-dimensional central cross section
$Y$ of $X$ that attains the approximation radius $z_{c_*,k}$. Let $r_k =
\min\{z_{c_*,k}(X),\sqrt{k}\sigma\}$, and $Y_k = Y\cap B_2(r_k)$. Clearly, $R(X,\sigma)\geq R(Y_k,\sigma)$, since $Y_k \subseteq X$.  We will lower bound $R(Y_k,\sigma)$ by applying Proposition~\ref{pro:Fano_mod} and Lemmas~\ref{lem:net}, \ref{lem:pack}. 

Since $\r(Y_k) \le r_k$, by Lemma~\ref{lem:net} for any $\ve \le r_k$, there exists an $\ve$-net of $Y_k$, say $N$, with $|N|\leq
(3r_k/\epsilon)^k$.  On the other hand, by Fact~\ref{fact:volumeratio}, 
$\vr_k(Y,r_k) \geq \vr_k(Y,z_{c_*,k}(X)) = c_*$. Therefore
\[\vol_k(Y_k)=\vol_k(Y\cap B_2(r_k)) = \vr_k(Y,r_k)^k \vol_k(B^k_2(r_k)) \geq c_*^k \vol_k(B^k_2(r_k))\,.\]

Combining it with Lemma~\ref{lem:pack}, there exists a $\delta$-packing of $Y_k$, say $M$, with $|M|\geq c_*^k \vol_k(B^k_2(r_k))/\vol_k(B^k_2(\delta)) = (c_*\cdot r_k/\delta)^k$. 

Choose $\delta = (c_*/a)\, r_k$, and $\ve = r_k$, where $a$ is a
constant to be determined. Using the bounds on $|N|$ and
$|M|$ in Proposition~\ref{pro:Fano_mod}, we obtain
\begin{eqnarray}\label{eq:R_lowerbound}
R(Y_k,\sigma) \ge \frac{1}{4} \left(\frac{c_*r_k}{a} \right)^2 \Big( 1- \frac{k \log 3 + \frac{r_k^2}{2\sigma^2} + 1}{k\log a}\Big)
\ge \frac{1}{4} \left(\frac{c_*r_k}{a} \right)^2 \Big(1 - \frac{\log 3 + \frac{3}{2}}{\log a} \Big).\end{eqnarray}
Maximizing the right hand side over $ a >1 $, we get $a = 12.89$. Plugging in for $a$ in Eq.~\eqref{eq:R_lowerbound}, we obtain $R(Y_k,\sigma) \ge C c_*^2 r_k^2$, with $C= 2.46 \cdot10^{-4}$. Since $1\le k\le n$ is arbitrary, we have 
\begin{eqnarray*}
R(X,\sigma) \ge \max_k R(Y_k,\sigma)\ge \max_k C c_*^2r_k^2 = C c_*^2 \max_k \min \{z_{c_*,k}(X)^2, k\sigma^2\}.
\end{eqnarray*}
\end{proof}

Invoking relation~\eqref{eqn:upperbound}
and~Theorem~\ref{thm:lowerbound}, in order to prove the near
optimality of truncated series estimators for family $\F^{m,n}_{\infty}$,
we establish some properties of the Kolmogorov width and explore its
relation to the approximation radius. Before proceeding, we make a comparison between the proposed lower bound, and the one obtained by considering the hardest rectangular sub-problem. 
\smallskip\\

\noindent {\bf Relation to the hardest rectangular sub-problem.}
One technique in the literature~\cite{dlm-90,nemirovski-98} for lower
bounding the minimax risk is to find the ``hardest'' box contained in
the body (or compute the \emph{Bernstein width}, defined as the side
length of the largest cube enclosed in the body) and apply the known
lower bound for the box.  The approximation radius can always be used
to achieve at least the same asymptotical lower bound. 

Suppose that $X$ contains a box with side lengths $\tau_1,\ldots, \tau_n$.  Then using the box bound~\cite{dlm-90}, we have that $R(X,\sigma) =  \Omega (\sum_i \tau_i^2 \sigma^2/(\tau_i^2+\sigma^2)) = \Omega(\sum_i
\min(\tau_i^2,\sigma^2))$. Now group $\tau_i$'s as follows. The first group consists of $\tau_1,\cdots,\tau_{k_1}$, where $k_1$ is the smallest index such that $\sum_{j=1}^{k_1} \min(\tau_j^2,\sigma^2) \ge \sigma^2$. The seconds group consists of $\tau_{k_1+1}, \cdots, \tau_{k_2}$, where $k_2$ is the smallest number such that $\sum_{j=k_1+1}^{k_2} \min(\tau_j^2,\sigma^2) \ge \sigma^2$, and so forth. Let $k$ be the total number of groups. Firstly, note that $\sum_{i\in I} \min(\tau_i^2,\sigma^2)$ is at most $2 \sigma^2$, for all groups $I$. Hence $k = \Omega (\sum_{i} \min(\tau_i^2/\sigma^2,1))$. Secondly, by construction, for all groups $I$ (except possibly the last one), we have $\sum_{i\in I} \min(\tau_i^2,\sigma^2) \ge \sigma^2$. Let $k'$ be the number of these groups. For each of them, we can replace the corresponding face
by its diagonal with length $\sqrt{\sum_{i\in I} \tau_i^2}\geq
\sigma$.  This way we obtain an $k'$ dimensional box with each
side length at least $\sigma$. Now it is straight forward to see that,
$z_{c,k'}(X) = \Omega(\sqrt{k'}
\sigma)$, and by Theorem~\ref{thm:lowerbound}, we get a lower bound of $\Omega(k' \sigma^2) = \Omega(k\sigma^2) = \Omega(\sum_{i} \min(\tau_i^2,\sigma^2))$. 

\section{A duality relationship for Kolmogorov widths}\label{sec:dual_kol}
We take a detour to establish the connection between the Kolmogorov
width and the approximation radius. The connection is via a novel
duality relationship between the Kolmogorov widths of $X$ and those of
its polar dual, as stated in Theorem~\ref{thm:dual}.  The proof is an
application of some celebrated works in convex
geometry~\cite{bt-87,st-89,g-95}.

\begin{definition} A set of vectors
$V=\{v_1, \cdots, v_s\}$ is called $\delta$-wide if for any $1\leq
i\leq s$, $\dist(v_i,\span[V/\{v_i\}])\geq \delta$. 
\end{definition}

The following proposition concerns an interesting property of
$\delta$-wide sets, and can be gleaned from~\cite{bt-87,st-89,g-95}. For reader's convenience, we give the proof
of this proposition in Appendix~\ref{app:w}.
\begin{propo}\label{pro:w}
For any $\delta$-wide set $V = \{v_1,\cdots,v_s\}$, there exists $\sigma \subseteq
\{1,\ldots, s\}$ with $\size{\sigma}\geq (1-\epsilon) s$ such that for
any $\alpha = (\alpha_j)_{j \in \sigma}$,
\begin{eqnarray*}
\|\sum_{j\in \sigma} \alpha_j v_j \| \geq c \sqrt{\frac{\epsilon}{s}} \delta \sum_{j\in \sigma} |\alpha_j|\,,
\end{eqnarray*}
with $c = (\sqrt{2}-1)/2$.
\end{propo}

Now we use the above proposition to prove Theorem~\ref{thm:dual}.
\begin{proof}(Theorem~\ref{thm:dual})
Write  $\delta = d_{k}(X)$. 
Consider the $k+1$ points $V=\{v_1, \ldots, v_{k+1}\}$ inside $X$
which forms the largest $k+1$ simplex. By the maximality of the volume
of $V$, for any $1\leq i\leq k+1$,
\begin{eqnarray}\label{eqn:maximal1}
\dist(v_i , \span[V/\{v_i\}]) = \max_{x\in X} \dist(x, \span[V/\{v_i\}])\,.
\end{eqnarray}
Note that the vectors in $V$ are affinely independent, and thus $\dim(V/\{v_i\})$ is either $k$ or $k-1$. Therefore, there exists an $r$-codimensional  projection $P$ such that $\ker(P) = V/\{v_i\}$, and $r \in\{ k-1, k\}$.  
Then
\begin{eqnarray}\label{eqn:maximal2}
\dist(v_i , \span[V/\{v_i\}]) = \| P(v_i) \|\,.
\end{eqnarray}

\noindent Also, by Eq.~\eqref{eqn:maximal1}, we have 
\begin{eqnarray}
\|P(x)\| = \dist(x, \span[V/\{v_i\}]) \leq \dist(v_i, \span[V/\{v_i\}]) = \|P(v_i)\|\,,
\end{eqnarray}
for any $x\in X$. On the other hand, since $d_r(X) \ge d_k(X) = \delta$ and $X$ is centrally symmetric, there exist $x,y\in X$, such that $\|P(x)-P(y)\|\geq 2\delta$. 
Hence 
\begin{eqnarray*}
\|P(v_i)\|\geq \frac{1}{2} (\|P(x)\|+\|P(y)\|) \geq \frac{1}{2} \|P(x)-P(y)\| \geq \delta\,.
\end{eqnarray*}

\noindent Using Eq.~\eqref{eqn:maximal2}, $V$ is $\delta$-wide. By Proposition~\ref{pro:w}, there exists $\sigma \subseteq \{1,\ldots, k\}$ with $\size{\sigma}\geq (1-\epsilon) k$ such that for any $\alpha = (\alpha_j)_{j \in \sigma}$,
\begin{eqnarray}
\|\sum_{j\in \sigma} \alpha_j v_j \| \geq c \sqrt{\frac{\epsilon}{k}} \delta \sum_{j\in \sigma} |\alpha_j|\,, \quad c = \frac{\sqrt{2}-1}{2}.
\end{eqnarray}

Let $H=\span[\{v_i\,|\,i\in \sigma\}]$. We claim that
\begin{eqnarray*}
 H\cap X \supseteq H \cap B^n_2(c\sqrt{\epsilon/k}\,\delta)\,.
 \end{eqnarray*}

\noindent Consider $Y=\{\sum_{j\in \sigma} \alpha_j v_j\,|\, \sum_{j\in \sigma} |\alpha_j|\le1\}$. Since $X$ is convex and centrally symmetric, and $\{v_i\} \subseteq H\cap X$, we have $Y\subseteq H\cap X$. Hence, it suffices to show that
$H\cap B^n_2(c\sqrt{\epsilon/k}\, \delta)\subseteq Y$. 
For any given  $x\in H\cap B^n_2(c\sqrt{\epsilon/k}\, \delta)$, let $r^\ast=\max \{r: rx\in Y\}$.  Clearly, there exists $\alpha = (\alpha)_{j \in \sigma}$, such that, $r^\ast x = \sum_{j\in \sigma} \alpha_j v_j$, and $\sum_{j\in \sigma} | \alpha_j | = 1$. Hence,
\begin{eqnarray}\label{eqn:maximal3}
\| r^\ast x\| = \| \sum_{j\in\sigma} \alpha_j v_j\| \geq c\sqrt{\frac{\epsilon}{k}}\delta\sum_{j\in\sigma}|\alpha_j| = c\sqrt{\frac{\epsilon}{k}}\delta\,.
\end{eqnarray}

\noindent As $x\in H\cap B^n_2(c\sqrt{\epsilon/k}\,\delta)$, we have $\|x\| \le c\sqrt{\epsilon/k}\,\delta$ and by Eq.~\eqref{eqn:maximal3}, we obtain $r^\ast\geq 1$. Consequently, $x\in Y$. Since $x \in H\cap B^n_2(c\sqrt{\epsilon/k}\, \delta)$ was arbitrary, we have 
\begin{eqnarray}\label{eqn:maximal4}
H\cap B^n_2(c\sqrt{\epsilon/k}\, \delta)\subseteq Y \subseteq H\cap X.
\end{eqnarray} 

By Fact~\ref{fact:dual},
\begin{eqnarray*}
P_H(X^\circ) & = & (H\cap X)^\circ\\
& \subseteq & (H\cap B^n_2(c\sqrt{\epsilon/k}\,\delta))^\circ\\
& = & H\cap B_2 \left(\frac{1}{c\sqrt{\epsilon/k}\,\delta}\right)\,.
\end{eqnarray*}

\noindent Thus $\r(P_H(X^\circ))\leq 1/(c\sqrt{\epsilon/k}\,\delta)$. Note that $P_H \in \mathcal{P}_{n - \dim(H)}$. Hence,
\begin{eqnarray}\label{eqn:maximal5}
d_{n-\dim(H)}(X^\circ)  = \min_{P \in \P_{n-\dim(H)}} \r(P(X^\circ))\leq \frac{1}{c\sqrt{\epsilon/k}\,\delta}.
\end{eqnarray}

\noindent Since $\dim(H) = |\sigma| \ge (1-\epsilon)k$, recalling Eq.~\eqref{eqn:monotone_dk}, $d_{n- (1-\ve)k}(X^{\circ}) \le d_{n - \dim(H)}(X^{\circ})$. Taking $c_1=1/c = 2/(\sqrt{2}-1)$, we have 
\begin{eqnarray}\label{eqn:maximal6} 
d_k(X) d_{n-(1-\epsilon)k}(X^\circ) \leq c_1 \sqrt{\frac{k}{\epsilon}}\,.
\end{eqnarray}
\end{proof}

Before we pass to the next section, we make a few remarks about the duality relationship stated in Theorem~\ref{thm:dual}.

\begin{remark}\label{remark:tight}
The dependence on $k$ is the best possible. Consider $X=B_1^n$, the unit $\ell_1$ ball. Then $X^\circ = B_\infty^n$. It is easy to see that for any $0\leq k,k'\leq n$, $d_k(X)\geq \sqrt{1-k/n}$ and $d_{n-k'}(X^\circ) \geq \sqrt{k'}$. When $k\leq n/2$ and $k'=\Omega(k)$, we have that $d_k(X)d_{n-k'}(X^\circ) = \Omega(\sqrt{k})$.  We do not know if the dependence on $\epsilon$ is the best possible. But for the application in this paper, the dependence on $\epsilon$ is not significant as it will be chosen as a constant. 
\end{remark}

\begin{remark}\label{remark:john}
By using the maximum volume ellipsoid, it is fairly easy to show that for any $0\leq k<n$, 
\begin{eqnarray*}
d_k(X)d_{n-k-1}(X^\circ)\leq \sqrt{n}\,.
\end{eqnarray*}

Consider the maximal enclosed ellipsoid $E\subseteq X$. By John's theorem~\cite{John-48}, $E\subseteq X\subseteq \sqrt{n} E$.  Let the axes lengths of $E$ be $\lambda_1\geq \lambda_2\ldots\geq \lambda_n\geq 0$. Then $d_k(X)\leq \sqrt{n}\lambda_{k+1}$ since $X\subseteq \sqrt{n}E$. On the other hand, by duality $X^\circ \subseteq E^\circ$.  The axes lengths of $E^\circ$ are $1/\lambda_n\geq\ldots\geq 1/\lambda_2 \geq 1/\lambda_1$.  So $d_{n-k-1}(X^\circ)\leq 1/\lambda_{k+1}$. Therefore, $d_k(X)d_{n-k-1}(X^\circ)\leq \sqrt{n}$.

However, proving the stated bound requires more advanced tool (Proposition~\ref{pro:w}) .
\end{remark}

\begin{remark}
In~\cite{m-90}, a duality about Gelfand numbers are given, where Gelfand number $c_k$ is defined as
\begin{eqnarray*}
c_k(X) = \min_{H: \codim(H)=k}\r(H\cap X)\,.
\end{eqnarray*}

Observe that $c_k\leq d_k$. To put it in a comparable form,
in~\cite{m-90}, it is shown that there exists constant $D>0$, such
that for any $0<\kappa<1$, $c_k(X) c_{(1-\kappa)n-k-D}(X^\circ) = O(1/\kappa)$.
This duality relation focuses on the duality gap, i.e. the product can be upper bounded by any constant. In our case there is a factor of $\sqrt{k}$. However, the dimension gap, i.e. the difference between the dimension in one term and the co-dimension in the other term, is $\kappa n$ in this relationship. But ours is $\epsilon k$, much smaller when $k$ is small. 
 If we were to apply the duality
in~\cite{m-90}, we then need to set $\kappa = \epsilon k/n$, resulting
in a bound of $O(n/(\epsilon k))$, much larger than our bound when $k$
is small.  But the duality in~\cite{m-90} holds with high probability
for a randomly chosen subspace, while it is not true for our bound.
\end{remark}

\section{Main theorem}

In the previous section, we showed a relationship between the Kolmogorov
widths of $X$ and its polar dual $X^\circ$.  This easily
translates to a relation between the Kolmogorov width and the radius
of the largest Euclidean ball contained in $X$, which in turn
gives us a bound on the minimax risk of the truncated series estimator.
However, the bound is fairly weak due to the large duality gap of
$\sqrt{k}$.  If it were some constant in place of $\sqrt{k}$, we would
already obtain the results we search after.  Unfortunately per
Remark~\ref{remark:tight}, this dependence cannot be improved.
In this section, we show that if $X$ is defined by $m$ hyperplanes, we can scale the largest
Euclidean ball contained in $X$ by the factor of $\sqrt{k/\log m}$ such that the fraction of its volume inside $X$ is still non-negligible, despite that the scaled ball may grow outside of $X$. This gives us
the proof of Theorem~\ref{thm:approxradius}, and therefore of
Theorem~\ref{thm:linear}. 

\begin{proof}(Theorem~\ref{thm:approxradius})
Let $X$ be an arbitrary element in $\mathcal{F}^{m,n}_{\infty}$. Hence, there exists $A \in \reals^{m \times n}$ such that $X = \{x \in \reals^n:  |Ax|_{\infty} \le 1\}$.
Let $\tau = d_{n-k} (X^\circ)$. By definition of Kolmogorov widths, there exists a subspace $H$ with $\dim(H) = k$, such that $\r(P_H(X^\circ)) = \tau$.  Let $H=(h_1, \ldots, h_k)$, where $h_i$'s are any orthonormal bases of $H$. 
By Fact~\ref{fact:dual0}, $X^\circ = A^{T}B^m_1$, and $P_H(X^\circ) = HH^T A^{T} B^m_1$. As $\r(P_H(X^\circ)) = \tau$, for any $y\in P_H(X^\circ)$, $\|y\|\leq \tau$. Equivalently, for any $w\in B^m_1$, $\|H^TA^Tw\|\leq \tau$. Let $F = A H$ and write $F=(f_{ij})_{m\times k}$. By duality of matrix norms,
\begin{eqnarray*}
\max_{w\in B_1} \|F^Tw\| = \max_{1 \le i \le m} \sqrt{\sum_{j=1}^{k} f_{ij}^2}.
\end{eqnarray*}
Since $\|F^Tw\|\leq \tau$ for $w\in B^m_1$, we have $\sqrt{\sum_j f_{ij}^2} \leq \tau$, for any $1\leq i\leq m$. 

Consider a random vector $g = (g_1, \ldots, g_k)$ where $g_i$'s are i.i.d. standard gaussians. Denote by $\mu$ the probability density function of $g$, i.e., 
\begin{eqnarray*}
\mu(g) = \frac{1}{(2\pi)^{k/2}} \exp\{-\frac{1}{2} \sum_{i=1}^k g_i^2\}.
\end{eqnarray*}
Let $\mu_0 = 1/(2\pi)^{k/2}$, $r = \sqrt{2k\ln(1/(2c_*))}$, and $\mu_1 = \mathbb{P}(\|g\| \le r)$.

Using the standard tail bound for sum of random normal variables~\cite{feller-72}, for any constant $c>0$, and for any $1\le i \le m$,
\begin{eqnarray}
\mathbb{P}\Big \{ |(Fg)_i| \geq c \sqrt{(\sum_{j=1}^k f_{ij}^2) \, \ln m}\Big\} \leq m^{-c^2/2} \,.
\end{eqnarray}

Since $\sqrt{\sum_j f_{ij}^2} \leq \tau$, we obtain
\begin{eqnarray}
\mathbb{P}\Big \{ |(Fg)_i| \geq c \tau \sqrt{\ln m}\Big\} \leq m^{-c^2/2} \,.
\end{eqnarray}

Applying union bound for $1 \le i \le m$, we obtain
\begin{eqnarray}
\begin{split}
\mathbb{P}\{F g \in c \tau \sqrt{\ln m}\, B^m_\infty\} =
1- \mathbb{P} \big(\cup_{i=1}^m \{ |(Fg)_i| \geq c \tau \sqrt{\ln m}\} \big)
\geq 1- m^{1-c^2/2}.
\end{split}
\end{eqnarray}

Consequently,
\begin{eqnarray}
\begin{split}
\mathbb{P}\{H g \in c\tau \sqrt{\ln m}\, X\}  &= \mathbb{P}\{|AH g|_{\infty} \le c\tau \sqrt{\ln m}\}\\
& = \mathbb{P}\{|F g|_{\infty} \le c\tau \sqrt{\ln m}\} \ge 1 - m^{1-c^2/2}.
\end{split}
\end{eqnarray}

 Assuming $m\ge 2$, and letting $c = \sqrt{4-2\log_2 \mu_1}$, we obtain $\mathbb{P}\{H g \in c\tau \sqrt{\ln m}\, X\} \geq 1-\mu_1/2$.
Note that the function $\mu(g)$ is decreasing in $\|g\|$. Therefore,
\begin{eqnarray}\label{eqn:zk_bound1}
\begin{split}
\vol\Big( c \tau \sqrt{\ln m}\,X \cap B^k_2(r)\Big) & \ge \frac{1}{\mu_0} \mathbb{P} \Big\{Hg \in \Big(c\tau \sqrt{\ln m} \,X \cap B^k_2(r)\Big) \Big\}\\
& \ge \frac{1}{\mu_0} \Big( \mathbb{P} (Hg \in c\tau \sqrt{\ln m}\, X) + \mathbb{P} (Hg \in B^k_2(r)) - 1\Big)\\
&\ge \frac{1}{\mu_0} \Big(1 - \frac{\mu_1}{2} + \mu_1  - 1\Big)= \frac{\mu_1}{2\mu_0}\,.
\end{split}
\end{eqnarray}
Here, $B_2^k(r)$ is the $k$ dimensional $\ell_2$ ball in the subspace $H$. (Recall that $\dim(H) = k$).
\begin{fact}\label{fact:mu1}
Let $\mu_0 = 1/(2\pi)^{k/2}$, $r = \sqrt{2k\ln(1/(2c_*))}$, and $\mu_1 = \mathbb{P}(\|g\| \le r)$.
The followings hold true.

\noindent $(a)\,$ $\mu_1\ge \mu_0\,e^{-r^2/2} \vol(B^k_2(r))$.
\vspace{.4cm}

\noindent $(b)\,$ If $0 < c_* \leq 0.2$, then
\begin{eqnarray}
\mu_1 \ge 1 - 2c_*\sqrt{2e\ln\frac{1}{2c_*}} \ge 0.1\,.
\end{eqnarray}
\end{fact}
We refer to Appendix~\ref{app:mu1} for the proof of Fact~\ref{fact:mu1}.

Using Fact~\ref{fact:mu1} (part $(a)$) in Eq.~\eqref{eqn:zk_bound1}, we get
\begin{eqnarray}\label{eqn:zk_bound2}
\bigg\{\frac{\vol\big( c \tau \sqrt{\ln m}\,X \cap B^k_2(r)\big)}{\vol(B^k_2(r))} \bigg\}^{\frac{1}{k}}
\ge \frac{1}{2^{1/k}} e^{-r^2/2k} = \frac{2c_*}{2^{1/k}}\ge c_*\,.
\end{eqnarray}
Scaling the sets by factor $1/(c\tau \sqrt{\ln m})$ in the left hand side of Eq.~\eqref{eqn:zk_bound2}, and using the definition of approximation radius,
\begin{eqnarray}
z_{c_*,k}(X)\geq \frac{r}{c\tau \sqrt{\ln m}}  = c_2 \sqrt{\frac{k}{\ln m}} \cdot \frac{1}{d_{n-k}(X^\circ)}\,,
\end{eqnarray}
where $c_2 = (\sqrt{2}/c) \sqrt{\ln(1/(2c_*))}$. Using Fact~\ref{fact:mu1} (part $(b)$), 
\begin{eqnarray}
c = \sqrt{4 - 2 \log_2 \mu_1} \le 3.3\,,
\end{eqnarray} 
whence we obtain $c_2 \ge 0.4 \sqrt{\ln(1/(2c_*))}$. This concludes the proof.
\end{proof}

With all these preparations, we can now prove the main theorem.

\begin{proof}(Theorem~\ref{thm:linear})
As mentioned earlier, the lowerbound is implied by previous work. We
only show the upperbound. Recall that $d_k(X)$ is non-decreasing
in $k$. (see Eq.~\eqref{eqn:monotone_dk}). Let $k^\ast = \min \{ k\ge
1 | d_k(X)^2 \leq k \sigma^2 \}$. ($k^*$ exists since $d_n(X) =
0$). Consider the two cases below separately.
\smallskip\\

\noindent$\bullet$ $\mathbf{(k^* > 1) : }$ Invoking Eq.~\eqref{eqn:upperbound}, $R_T(X,\sigma) \le d_{k^*}(X)^2 + k^* \sigma^2$. By definition of $k^*$, $R_T(X,\sigma) \le 2 k^* \sigma^2$. Further,
\begin{eqnarray*}
\begin{split}
d_{k^*}(X)^2 + k^* \sigma^2 &\le d_{k^*-1}(X)^2 + \frac{k^*}{k^*-1}\,(k^*-1)\sigma^2\\ 
&\le d_{k^*-1}(X)^2 + 2 d_{k^*-1}(X)^2 = 3d_{k^*-1}(X)^2. 
\end{split}
\end{eqnarray*}
Hence, $R_T(X,\sigma) \le 3 \min \{d_{k^\ast-1}(X)^2, k^\ast\sigma^2\}$. On the other hand,
\begin{eqnarray}
\begin{split}
z_{c_*,(k^*-1)/2}(X) &\ge c_2 \sqrt{\frac{k^*-1}{2\ln m}} \cdot \frac{1}{d_{n-(k^*-1)/2}(X^\circ)}\\
&\ge \frac{c_2}{2c_1 \sqrt{\ln m}} d_{k^*-1}(X),
\end{split}
\end{eqnarray}
where the first inequality follows from Theorem~\ref{thm:approxradius} and the second one follows from Theorem~\ref{thm:dual}. Applying Theorem~\ref{thm:lowerbound},
\begin{eqnarray}
\begin{split}
R(X,\sigma)&\ge C c_*^2\min \Big\{z_{c_*,(k^*-1)/2}(X)^2, \frac{k^*-1}{2} \sigma^2\Big\}\\
& \ge C c_*^2\min\Big\{\frac{c_2^2}{4c_1^2 \ln m} d_{k^\ast - 1}(X)^2, \frac{ k^\ast - 1}{2} \sigma^2\Big\}\\
& \ge \frac{C c_*^2}{4\ln m} \min(c_2^2/c_1^2,1) \min\Big\{d_{k^\ast - 1}(X)^2, k^* \sigma^2\Big\}\\
&\ge \frac{C_1}{\ln m} R_T(X,\sigma),
\end{split}
\end{eqnarray}
for $C_1 = (C c_*^2/12) \min(c_2^2/ c_1^2,1)$. 
\smallskip\\

\noindent$\bullet$ $\mathbf{(k^*  = 1) : }$ Using Eq.~\eqref{eqn:upperbound},
\begin{eqnarray}
R_T(X,\sigma) \le \min \{d_0(X)^2, d_1(X)^2 + \sigma^2\} \le \min\{\r(X)^2, 2\sigma^2\},
\end{eqnarray}
where we used the assumption $k^* = 1$ in the final step. On the other hand, $X$ contains a segment $S$ with length $\r(X)$. Using the result of~\cite{dlm-90}, 
\begin{eqnarray}
R(X,\sigma) \ge R(S,\sigma) = \frac{\sigma^2 \cdot \r(X)^2}{\sigma^2+\r(X)^2} \ge \frac{1}{2} \min(\sigma^2,\r(X)^2).
\end{eqnarray}
Therefore, $R(X,\sigma)\ge (1/4) R_T(X,\sigma)$.

Combining both cases, we have
\begin{eqnarray}\label{eq:main_thm_proof}
\frac{R_T(X,\sigma)}{R(X,\sigma)} \le M_{c_*} \ln m\,,
\end{eqnarray}
where
\begin{eqnarray*}
M_{c_*} = \frac{1}{C_1} = \frac{12}{C c_*^2} \max (c_1^2/c_2^2, 1), \; C = 2.46 \cdot 10^{-4}\;, c_1 = 2/(\sqrt{2}-1)\;, c_2 = 0.4 \sqrt{\ln(1/(2c_*))}.
\end{eqnarray*}
Minimizing $M_{c_*}$ over $0 < c_*\le 0.2$, we obtain $c_* = 0.2$ with $M_{c_*} < 2 \cdot 10^{8}$.
\end{proof}

\begin{remark}
It is essential that $X$ is symmetric. Otherwise, we can take an
orthant of $B^n_1$ which has $O(n)$ faces and has large gap between
$R_T(X,\sigma)$ and $R(X,\sigma)$. 
\end{remark}

\section{Discussions}

\subsection{Applications to estimating Lipschitz functions}

The problem of estimating values of a Lipschitz function, at a set of
sampled points, from noisy measurements is discussed in the
introduction. Since the Lipschitz condition can be represented as
linear conditions, Theorem~\ref{thm:linear} is widely applicable to
such problems. For example, the function can be defined on any metric
space, the sampling points can be arbitrary set of points, and the
Lipschitz condition can be of higher order. As long as the
corresponding linear constraints is bounded by $n^{O(1)}$ for $n$
samples, the approximation factor is within a small factor of $O(\log
n)$ of the optimal.

\subsection{Smooth convex bodies}\label{sec:smooth}

In the above, we have shown that $\beta^{m,n}_\infty = O(\log m)$.
The celebrated Pinsker bound~\cite{pinsker-80} states that
$\beta^{m,n}_2 = O(1)$.  What about $\beta^{m,n}_p$ for other $p$'s?
By plugging $\sigma=1/\sqrt{n}$ in Theorem~3 in~\cite{dj-94}, we
have that for $1\leq p<2$, $\beta^{n,n}_p=\Omega((n/\log n)^{1-p/2})$.
So we will not be able to obtain a similar bound to
Theorem~\ref{thm:linear} when $p<2$.  On the other hand, we conjecture
that similar upperbound holds when $p\geq 2$.  
\begin{conjecture}
For any $p\ge 2$, there exists a constant $C = C(p)$, such that for
any $m,n \ge 2$, $\beta^{m,n}_p \leq C \log m$. 
\end{conjecture}

Define the distance $d(X,Y)$ between two centrally symmetric convex
body $X,Y$ as the smallest $c$ such that there exists a uniformly scaled orthogonal 
transformation $F$ such that $FY \subseteq X \subseteq c FY$. We note that $d(\cdot,\cdot)$ is similar to but different from the classical Banach-Mazur distance in which $F$ is any linear transformation. 
and that $\log d(\cdot,\cdot)$ is a pseudometric (non-negative, symmetric, and with triangular inequality).
By straightforward arguments, $\beta(X)\leq d(X,Y)^2\beta(Y)$.  Since $d(B^n_p, B^n_2) = n^{1/2 - 1/p}$ and $d(B^n_p,
B^n_\infty) = n^{1/p}$, we have the following nontrivial bound.
\begin{corollary}
For $p\geq 2$, $\beta^{m,n}_p = O(\min(n^{1-2/p}, m^{2/p}\log m))$.
In particular, for $p\geq 2$, $\beta^{n,n}_p = O(\sqrt{n\log n})$.
\end{corollary}

\subsection{Tightness of the approximation radius bound}

We have used the approximation radius to lower bound the minimax risk
of a convex body $X$. How tight is this bound? This paper has shown
that it is at least within $O(\log m)$ factor of the optimal upper
bound, and it is achieved by using the (rather limited) truncated series
estimators.

As discussed before, the approximation radius provides a
lower bound at least as good as using Bernstein width, which is known
to be asymptotically optimal for $B^n_p$ when $p\geq 2$. In this
section, we consider $B^n_p$ for $1\leq p<2$ and show that the lower
bound of using approximation radius is very close to the minimax upper
bound but does leave a small gap of factor of $\Theta((\log n)^{1-p/2})$.

We start by upper bounding $z_{c,k}(X)$.
For any linear $k$-dimensional subspace $H_k$, and $B^k_2(r) \subset H_k$, we have 
\begin{eqnarray}
B^k_2(r) \cap B^n_p \subseteq H_k \cap B^n_p.
\end{eqnarray}
As it is proved in~\cite{Meyer-88}, if $1\le p \le 2$, then $\vol(H_k \cap B^n_p) \le \vol(B_p^k)$. Using the formula for the volume of $k$-dimensional $\ell_p$ ball~\cite{wang-05}, we have
\begin{eqnarray}
\vol(B^k_p) = 2^k \frac{\Gamma^k(1+\frac{1}{p})}{\Gamma(\frac{k}{p}+1)} = \left(\frac{C_p}{k^{1/p}}\right)^k,
\end{eqnarray}
where $C_p$ is a constant that depends on $p$. Hence, for any $H_k$,
\begin{eqnarray}
\left(\frac{\vol(B^k_2(r)\cap B_p^n)}{\vol(B^k_2(r))}\right)^{1/k} \le \left(\frac{C_p^k}{C_2^k} \cdot \frac{k^{k/2}}{k^{k/p} r^k} \right)^{1/k} = \frac{C_p}{C_2} \frac{k^{1/2-1/p}}{r}.\label{eq:volume}
\end{eqnarray}

Therefore, $z_{c,k}(B^n_p) \leq \frac{C_p}{C_2\cdot
c}k^{1/2-1/p}$. For the lower bound of $z_{c,k}(B^n_p)$, choose $H_k$ to be one of the $k$-dimensional principal subspaces. Then $B^n_p\cap H_k = B^k_p\supset k^{1/2-1/p} B^k_2$.  Hence, $z_{c,k}(B^n_p)\geq \frac{1}{c} k^{1/2-1/p}$. So, $z_{c,k} = \Theta(k^{1/2-1/p}/c)$. 
 Apply the lower bound in
Theorem~\ref{thm:lowerbound} and we obtain $R(B^n_p,\sigma) = \Omega(\max_k
\min(k^{1-2/p}, k\sigma^2))$. When $\sigma\leq 1$, we choose $k\approx \sigma^{-p}$ and obtain a lower bound of $R(B^n_p,\sigma) = \Omega(\sigma^{2-p})$.  By~\cite{dj-94}, the optimal upper bound for $B^n_p$ is 
$R=\Theta(\sigma^{2-p}(2\log n \sigma^p)^{1-p/2})$ for
$(1/n)^{1/p} \ll \sigma \ll \sqrt{1/\log n}$.  Hence the
approximation radius bound leaves a gap of $\Theta((\log
n)^{1-p/2})$. Actually, the largest gap we know of is $\sqrt{\log n}$
by setting $p=1$ in the above bound.  
\commented{\note{Li: I don't have as much as faith to call it a conjecture.} We conjecture this is the
largest possible gap.
\begin{conjecture}
For any $0<c\leq 1$, there exists a constant $f(c)>0$, such that for any centrally symmetric convex $X\subseteq \reals^n$, $R(X,\sigma)
\leq \sqrt{\log n} f(c) \max_k \min(z_{c,k}(X)^2, k\sigma^2)$. 
\end{conjecture}}

\subsection{Computational complexity}

We have shown that the truncated series estimator is close to optimal
for symmetric convex polytopes.  For the family of ellipsoids $\F_2^{m,n}$, the optimal truncated
series estimator can be computed by using the singular value
decomposition.  However, computing the best truncated series
estimator, or the Kolmogorov width, for symmetric convex polytopes, is
a hard problem.  When $k=0$, $d_0(X)$ is the diameter of $X$, and it
is exactly the $\ell_2$-norm maximization problem considered
in~\cite{b-02}.  The problem is NP-hard. Further, it is shown in~\cite{b-02} that it is hard to approximate within any constant factor unless P=NP. 

On the other hand, by using semi-definite programming (SDP)
relaxation, one can compute $O(\sqrt{\log m})$ approximation of the
diameter~\cite{gw-95,nrt-99}, i.e. $d_0(X)$. However, it is not known
how to approximate $d_k(X)$ for $k>1$. \cite{vvyz-07} showed that if
the number of vertices of $X$ is $v$, then SDP gives an $O(\sqrt{\log
v})$ approximation for $d_k$. However, in our problem, the
number of vertices of a symmetric convex body could be exponential in
$n$.  So the technique in~\cite{vvyz-07} does not directly apply to
our problem.

\section{Conclusion}

In this paper, we show that the truncated series estimator can achieve
nearly optimal minimax risk for symmetric convex bodies defined by few
hyperplanes.  There are some outstanding open questions raised by this
work.

\begin{enumerate}
\item What is the best bound for $\beta^{m,n}_\infty$? Our work leaves a gap of $\Omega(\sqrt{\log m/\log\log m})$ and $O(\log m)$.

\item What is the best bound for $\beta^{m,n}_p$ for $p\geq 2$? We conjecture it is $O(\log m)$.

\item How tight is the approximation radius bound for lower bounding the minimax risk for convex bodies? For $\ell_1$ ball, it has a gap of $\Theta(\sqrt{\log n})$. This is the largest gap we know of. 

\item How to efficiently approximate the optimal truncated series estimator for any symmetric convex polytope?
\end{enumerate}

\appendix
\section{Proof of Proposition~\ref{pro:Fano_mod}}
\label{app:Fano_mod}
Consider any $M_\delta(X)$- packing in $X$. Let $M_\delta(X) =
\{x_1,\cdots,x_r\}$, and let $u$ be a random variable uniformly
distributed on the hypothesis set $\{x_1,\cdots,x_r\}$. Denote by
$M(y)$, the estimation of $x$ given the observation $y$. Define $w =
{\rm argmin}_{1\le j \le n} \|M(y) - x_j\|$. Since $\|x_j - x_j'\| \ge
\delta$, we have $w = j$, if $\|M(y) - x_j\| \le \delta/2$. Therefore,
\begin{eqnarray}\label{eqn:Fano1}
\begin{split}
\max_{1\le j \le n} \E_{p_{x_j}} \|M(y) - x_j\|^2 &\ge \left(\frac{\delta}{2}\right)^2 \max_{1\le j \le n} \mathbb{P}\{\|M(y) - x_j\| \ge \frac{\delta}{2} | u = j\}\\
&\ge \frac{\delta^2}{4r} \sum_{j=1}^r \mathbb{P}(w \ne j | u = j)\\
&\ge \left(\frac{\delta}{2}\right)^2 \mathbb{P}(w \ne u).
\end{split}
\end{eqnarray}

Let $h(p)$ be the entropy function defined as
\begin{eqnarray*}
h(p) = -p \log p -(1-p) \log (1-p), \quad \text{for } 0\le p \le 1.
\end{eqnarray*}

Denote by $H(u|w)$ the posterior entropy of $u$, given $w$, and denote by $I(u;w)$ the mutual information between $u$ and $w$ defined as
\begin{eqnarray*}
I(u;w) = H(u) - H(u|w) = \log r - H(u|w).
\end{eqnarray*}

Using Fano's inequality~(\cite{Cover-91}, p. 39),
\begin{eqnarray}\label{eqn:Fano2}
\begin{split}
\mathbb{P}(w \ne u) \log(r-1) &\ge H(u|w) - h(1/2)\\
 &= H(u) - I(u;w) - \log 2\\
 &\ge \log r  - I(u;w) - \log 2.
\end{split}
\end{eqnarray}

We recall the definition of K-L distance between two probability densities $p,q$ on a set $\Omega$, defined as~\cite{Cover-91},
\begin{eqnarray}\label{eqn:KL_def}
D_{KL}(p,q) = \int p \log\frac{p}{q} d\mu,
\end{eqnarray}
where $\mu$ is any measure on $\Omega$.

Using a property of mutual information, and its relation to K-L divergence (~\cite{Cover-91}, p. 30, 33), we have
\begin{eqnarray}\label{eqn:Fano3}
\begin{split}
I(u;w) = I(u;M(y)) \le I(u;y) &= \E_{u}\{D_{KL}(P(y|u), P(y))\}\\
&\le \max_{1\le j \le r} D_{KL}(P(y|x_j), P(y))
\end{split}
\end{eqnarray}

Let $N_{\ve}(X)$ be any $\ve$-net for $X$. Considering the uniform prior distribution on $N_{\ve}(X)$, we write, $P(y) = 1/|N_{\ve}(X)|\, \sum_{\tilde{x} \in N_{\ve}(X)} P(y|\tilde{x})$. Also, by definition of $\ve$-net, for any $x_j$, $1\le j\le r$, there exists $\tilde{x}_j \in N_{\ve}(X)$, with $\|x_j - \tilde{x}_j\| \le \ve$. Hence,
\begin{eqnarray}\label{eqn:Fano4}
\begin{split}
D_{KL}(P(y|x_j), P(y)) &= \E \Big\{ \log \frac{P(y|x_j)}{\frac{1}{|N_{\ve}(X)|} \sum_{\tilde{x} \in N_{\ve}(X)} P(y|\tilde{x})} \Big\}\\
&\le \E \Big\{ \log\frac{P(y|x_j) }{\frac{1}{|N_{\ve}(X)|} P(y|\tilde{x}_j)} \Big\}\\
 &= \log |N_{\ve}(X)| + D(P(y|x_j), P(y|\tilde{x}_j))
\end{split}
\end{eqnarray} 

Following the model~\eqref{eqn:model}, $y|x_j \sim \normal(x_j,\sigma^2\I)$, and $y|\tilde{x}_j \sim \normal(\tilde{x}_j,\sigma^2\I)$. Using the definition of K-L distance (Eq.~\eqref{eqn:KL_def}), after some simple algebraic manipulations, we have 
\begin{eqnarray}\label{eqn:KLnormal}
D(P(y|x_j), P(y|\tilde{x}_j)) = \frac{1}{2\sigma^2} \|x_j - \tilde{x}_j\|^2 \le \frac{\ve^2}{2\sigma^2}.
\end{eqnarray}

Combining Eqs.~\eqref{eqn:Fano3},\eqref{eqn:Fano4}, and~\eqref{eqn:KLnormal}, we obtain
\begin{eqnarray}\label{eqn:Fano5}
I(u;w) \le \log |N_{\ve}(X)| + \frac{\ve^2}{2\sigma^2}.
\end{eqnarray}

Using Eq.~\eqref{eqn:Fano1},~\eqref{eqn:Fano2}, and~\eqref{eqn:Fano5}, we obtain the desired result.
\section{Proof of Lemma~\ref{lem:net}}
\label{app:net}
Since $r \ge \r(X)$, $X \subseteq B_2(r)$. Hence, any $\ve$-net for
$B_2(r)$ is also an $\ve$-net for $X$. We begin by covering $B_2(r)$ with a finite family of balls of radius $\epsilon$. Choose the sequence of centers $p_1,p_2,\cdots$ in such a way that
\begin{eqnarray*}
p_{i+1} \notin \bigcup_{j=1}^i B_2(p_j,\epsilon).
\end{eqnarray*}
When this is no longer possible, the sequence is terminated. Now the set $P =\{p_i\}$ is an $\epsilon$-net for $B_2(r)$. Meanwhile, note that the smaller balls $B_2({p_i},\epsilon/2)$ are all disjoint (since no two of the $p_i$ are within distance $\epsilon$ of each other). In addition, $B_2(p_i,\epsilon/2)\subseteq B_2(r) \oplus B_2(\epsilon/2)$, where $\oplus$ denotes the Minkowski sum. Therefore,
\begin{eqnarray}
\label{eqn:eps-net1}
|P| \vol ( B_2(\ve/2) ) = \sum_{p_i \in P} \vol \Big(B_2(p_i,\ve/2) \Big) \leq \vol \Big(B_2(r) \oplus B_2(\epsilon/2) \Big).
\end{eqnarray}
Evidently, $B_2(\epsilon/2)\subseteq 1/2\, B_2(r)$, since $\epsilon \le r$. Hence, $B_2(r)\oplus B_2(\epsilon/2)\subseteq 3/2\, B_2(r)$, and $\vol \Big( B_2(r)\oplus B_2(\ve/2) \Big) \leq (3/2)^n \vol \Big( B_2(r) \Big)$. Using Eq.~\eqref{eqn:eps-net1}, we obtain
 \begin{equation*}
 |P| \le \frac{(3/2)^n \vol \Big(B_2(r) \Big)}{\vol (B_2(\epsilon/2) )} \le \left(\frac{3 r}{\epsilon} \right)^n.
 \end{equation*}

\section{Proof of Lemma~\ref{lem:pack}}
\label{app:pack}
Let $M_{\delta}(X)$ denote the maximum size $\delta$-packing of $X$. By maximality of $M_{\delta}(X)$, any other point in $X$ is within $\delta$ distance of one of the points in $M_{\delta}(X)$. Hence, 
\begin{equation*}
X \subseteq \bigcup_{p \in M_{\delta}(X)} B_2(p,\delta),
\end{equation*}
whence we obtain
\begin{equation}
\label{eqn:pack1}
|M_{\delta}(X)| \ge \frac{\vol(X)}{\vol (B_2(p,\delta) )}.
\end{equation}

\section{Proof of Proposition~\ref{pro:w}}
\label{app:w}
The proof is based in a crucial way on the following lemma proved in~\cite{g-95}.
\begin{lemma}\label{lem:dummy}
Let $u_1,\cdots,u_s \in \reals^n$, $\|u_i\| \le 1$. Define the set
\begin{eqnarray*}
E = \{(\delta_j)_{j=1}^s : \|\sum_{j=1}^s \delta_j u_j\|^2 \le 2s\}.
\end{eqnarray*}
Then, for every $\ve \in (0,1)$, there exists $\sigma \subseteq \{1,\cdots, s\}$ with $|\sigma| \ge (1- \ve)s$, such that
\begin{eqnarray*}
P_{\sigma}(E) \supseteq c \sqrt{\ve}[-1,1]^{\sigma},\quad c = \frac{\sqrt{2}-1}{\sqrt{2}},
\end{eqnarray*}
where the restriction map $P_{\sigma}$ is defined as $P_{\sigma}:(\delta_j)_{j=1}^s \to (\delta_j)_{j \in \sigma}$.
\end{lemma}

Since the set $V = \{v_1,\cdots,v_s\}$ is $\delta$-wide, there exist $y_1, \cdots, y_s \in \reals^n$, so that 
\begin{eqnarray}\label{eqn:y_property}
\langle v_i,y_j \rangle = 1_{\{ i = j\}}, \text{\quad and \quad} \|y_i\| \le \frac{1}{\delta}, \quad i,j = 1,\cdots,s.
\end{eqnarray}

Let $u_i = \delta y_i$. Applying Lemma~\ref{lem:dummy}, there exists a set $\sigma\subseteq \{1,\cdots, s\}$, with $|\sigma| \ge (1- \ve)s$, and $P_{\sigma}(E) \supseteq c \sqrt{\ve} [-1,1]^{\sigma}$. Hence we can find $(\delta_{j})_{j=1}^s \in E$, such that, $\delta_j = c \sqrt{\ve} \sign(\alpha_j)$, for $j \in \sigma$. Then,
\begin{eqnarray}\label{eqn:dummy2}
\begin{split}
\sum_{j \in \sigma} |\alpha_j| &= \langle \sum_{j \in \sigma} \alpha_j v_j, \sum_{i =1}^{s} \sign(\alpha_i) y_i \rangle\\
 &= \frac{1}{c\sqrt{\ve}} \langle \sum_{j \in \sigma} \alpha_j v_j, \sum_{i =1}^{s} \delta_i y_i \rangle\\
 &\le \frac{1}{c\sqrt{\ve}} \Big\| \sum_{j \in \sigma} \alpha_j v_j\Big\| \cdot \frac{1}{\delta} \Big\|\sum_{i=1}^s \delta_i u_i \Big\|\\
 &\le \frac{1}{c \,\delta} \sqrt{\frac{2s}{\ve}} \Big\| \sum_{j \in \sigma} \alpha_j v_j\Big\|,
\end{split}
\end{eqnarray}
where the first step follows form Eq.~\eqref{eqn:y_property}. Rearranging the terms in Eq.~\eqref{eqn:dummy2} implies the result.

\section{Proof of Fact~\ref{fact:mu1}}
\label{app:mu1}
\begin{proof}[Proof (Part ($a$))]
\begin{eqnarray}
\begin{split}
\mu_1 &= \mathbb{P}(\|g\| \le r)
= \frac{1}{(2\pi)^{s/2}} \int_{\|x\| \le r} e^{-x^2/2} dx\\
&\ge \mu_0 \int_{\|x\| \le r} e^{-r^2/2} dx
= \mu_0\,e^{-r^2/2} \vol(B^s_2(r)).
\end{split}
\end{eqnarray}
\end{proof}
\begin{proof}[Proof (Part ($b$))]
We will first upper bound $\mathbb{P}(\|g\| > r)$ using a Chernoff Bound. 
\begin{eqnarray}\label{mu1_eq1}
\mathbb{P}(\|g\| > r) = \mathbb{P}(e^{\lambda \sum_{i=1}^k g_i^2} > e^{\lambda r^2}) \le \frac{\mathbb{E}\{ e^{\lambda \sum_{i=1}^k g_i^2}\}}{e^{\lambda r^2}}.
\end{eqnarray}
Since $g_i$ are i.i.d. standard normal variables, it is easy to see that
\begin{eqnarray}\label{mu1_eq2}
\mathbb{E}\{ e^{\lambda \sum_{i=1}^k g_i^2}\} = (\mathbb{E}\{e^{\lambda g_1^2}\})^k = \left( \frac{1}{\sqrt{1-2\lambda}}\right)^k.
\end{eqnarray}
Using Eq.~\eqref{mu1_eq2} in Eq.~\eqref{mu1_eq1} and substituting for $r$, we obtain
\begin{eqnarray}\label{mu1_eq3}
\mathbb{P}(\|g\| > r) \le \left( \frac{1}{\sqrt{1-2\lambda}} e^{-2\lambda \ln\frac{1}{2c_*}}\right)^k.
\end{eqnarray}
Minimizing the right hand side over $\lambda$ gives $\lambda  = 1/2(1+1/(2\ln(2c_*)))$. Notice that $\lambda > 0$, for $0 < c_* < 0.2$. Substituting for $\lambda$ in Eq.~\eqref{mu1_eq3} gives
\begin{eqnarray}\label{mu1_eq4}
\mathbb{P}(\|g\| > r) \le \left(2c_* \sqrt{2e \ln\frac{1}{2c_*}}\right)^k \le 2c_* \sqrt{2e \ln\frac{1}{2c_*}},
\end{eqnarray}
where the last step follows from $c_* \le 0.2$, and $k \ge 1$. Now, $\mu_1 = 1 - \mathbb{P}(\|g\| > r)$. The result follows.
\end{proof}
\bibliographystyle{abbrv}
\bibliography{minimax}
\end{document}